\documentclass[a4paper]{amsart}

\usepackage{latexsym, amssymb,amsmath, amsthm,amsfonts}

\newcommand{\kk}{\Bbbk}

\def\SLn{\operatorname{SL}_{n}(K)}
\def\SL{\operatorname{SL}}
\def\GLn{\operatorname{GL}_{n}(K)}

\def\SL2{\operatorname{SL}_{2}(K)}

\def\GL2{\operatorname{GL}_{2}(K)}
\def\Ga{{\mathbb G}_{a}}

\def\INVSL2{$K[V]^{operatorname{SL}_{2}(K)}$}
\def\INVSO2{$K[V]^{operatorname{SO}_{2}(K)}$}
\def\INVGL2{$K[V]^{operatorname{GL}_{2}(K)}$}

\def\GL{\operatorname{GL}}
\def\SL{\operatorname{SL}}

\def\ord{\operatorname{ord}}

\newtheorem{Lemma}{Lemma}[section]
\newtheorem{Theorem}[Lemma]{Theorem}

\newtheorem{Prop}[Lemma]{Proposition}

\theoremstyle{definition}
  \newtheorem{Def}[Lemma]{Definition}

\theoremstyle{remark}
  \newtheorem{rem}[Lemma]{Remark}

\newtheoremstyle{Acknowledgments}
  {}
    {}
     {}
     {}
    {\bfseries}
    {}
     {.5em}
     {\thmname{#1}\thmnumber{ }\thmnote{ (#3)}}
\theoremstyle{Acknowledgments}

\title[Basic $\Ga$-actions]
{Separating invariants for the basic $\Ga$-actions}

\author{Jonathan Elmer}
\address{University of Bristol\\
University Walk, Bristol\\
BS8 1TW}
\email{j.elmer@bris.ac.uk}

\author{Martin Kohls}
\address{Technische Universit\"at M\"unchen \\
 Zentrum Mathematik-M11\\
Boltzmannstrasse 3\\
 85748 Garching, Germany}
\email{kohls@ma.tum.de}

\date{\today}

\begin{document}


\begin{abstract}{We explicitly construct a finite set of separating invariants for the basic
  $\Ga$-actions. These are the finite dimensional 
indecomposable rational linear representations of the additive group $\Ga$ of
a field of characteristic zero,
and their invariants are the kernel of the
Weitzenb\"ock derivation $D_{n}=x_{0}\frac{\partial}{\partial{x_{1}}}+\ldots+
x_{n-1}\frac{\partial}{\partial{x_{n}}}$.\\

\noindent {\bf Keywords:} Invariant theory, separating invariants,  binary forms,
locally nilpotent derivations, basic $\Ga$-actions, generalized hypergeometric series.\\
{\bf AMS Classification:} 13A50, 13N15
}\end{abstract}


\maketitle
\section{Introduction}

A great many mathematical problems are special cases of the following:
let $K$ be a field of arbitrary characterstic and let $G$ be any
group. Suppose $G$ acts on the $K$-vector space $V$ and that $v$ and $w$ are
points of $V$. Is there a $g \in G$ satisfying $gv=w$? In other words, are $v$
and $w$ contained in the same $G$-orbit? Important examples include the case
where $G = \GLn$ acts on the vector space $V$ of $n \times n$ matrices by
conjugation, and the case where $G = \SLn$ acts on the space $V$ of binary
forms of degree $n$. The classical approach to these problems is to construct ``invariant polynomials''. These are polynomial functions $f: V \rightarrow K$ which satisfy $f(v)=f(gv)$ for all $g \in G$ and $v \in V$, and so are constant on $G$-orbits. In fact, one can define an action of $G$ on the set of polynomial functions $K[V]$ via $(g \cdot f)(v):= f(g^{-1}v)$ for which the invariant polynomials are the fixed points, $K[V]^G$, and these form a subalgebra of $K[V]$ called the \emph{algebra of invariants}. Ideally, one would like to find a complete set of algebra generators of $K[V]^G$, then use this set to distinguish as many orbits as possible.

This approach is not without its difficulties. For instance, it is not always
possible to distinguish all the orbits using invariant polynomials. As an
example, consider once more the case where $\GLn$ acts on the vector space
of $n \times n$ matrices over a field $K$ by conjugation. Provided $K$ is an infinite
field, the invariants are generated by the coefficients of the characteristic
polynomial \cite[Example~2.1.3]{DerksenKemper}, but it is well known that a
pair of matrices with the same characteristic polynomial are not necessarily
conjugate. More problematically, if $G$ is not reductive then the algebra
$K[V]^G$ may not even be finitely generated. Even if it
is, finding a set of generators can be a very difficult problem. If, however,
one is only interested in invariants from the point of view of separating
orbits, then finding a complete set of generators is not always necessary. It
is perhaps surprising, then, that Derksen and Kemper  made the following
defininition \cite[Definition 2.3.8]{DerksenKemper} as recently as 2002. 
\begin{Def}
A \emph{separating set} for the ring of invariants $K[V]^G$ is a subset $S \subseteq K[V]^G$ with the following property: given $v, w \in V$, if there exists an invariant $f$ satisfying $f(v) \neq f(w)$, then there also exists $s \in S$ satisfying $s(v) \neq s(w)$.    
\end{Def}

There are many instances in which separating sets can be seen to have
``nicer'' properties than generating sets. For example, it is well known that
if $G$ is finite and the characteristic of $K$ does not divide  $|G|$, then
$K[V]^G$ is generated by elements of degree $\leq |G|$, see \cite{FleischmannNoetherBound,FogartyNoetherBound}, but this is not
necessarily true in the modular case. On the other hand, the analogue for
separating invariants holds in arbitary characteristic
\cite[Theorem~3.9.13]{DerksenKemper}. Meanwhile, even if $K[V]^G$ is not
finitely generated, it is guaranteed to contain a finite separating set
\cite[Theorem~2.3.15]{DerksenKemper}. The existence proof is non-constructive,
which raises the question how to actually construct separating sets. Kemper
\cite{KemperCompRed} gives an algorithm for reductive groups, but using
Gr\"obner bases it is only effective for ``small'' cases. An example of a
finite separating set for a non finitely generated invariant ring is given in
\cite{DufresneKohlsFiniteSep}. 
For finite groups, a separating set can always be obtained as the coefficients
of a rather large polynomial \cite[Theorem 3.9.13]{DerksenKemper}. With
refined methods, ``nicer'' separating sets have been obtained for several
classes of finite groups and representations, see for example
\cite{SezerSeparating}. This paper goes in the same
direction: for the basic actions of the additive group in
characteristic zero, we present a rather small separating set. See also
\cite{DomokosSep, DraismaSeparating, ElmerSeparating,KohlsKraft}
for a small selection of other recent publications in the area.

From this point onwards, $\kk$ denotes a field of characterstic zero. In this
article we will concentrate on linear actions of the additive group $\Ga$ of
the ground field~$\kk$. 
The finite dimensional
indecomposable rational linear representations of $\Ga$ are called the basic $\Ga$-actions. There is one such action in each dimension, and these are described below:

\begin{Def}
Let $X_n:= \langle x_0, \ldots , x_n \rangle_{\kk}$ be a vector space of
dimension $n+1$. Then $\Ga$ is said to act \emph{basically} on $X_n$ (with
respect to the given basis) if the action of $\Ga$ on $X_{n}$ is given by the formula
\[a*x_i = \sum_{j=0}^i \frac{a^{j}}{j!}x_{i-j}, \qquad \text{ for all }a \in
\Ga, \,\,\, i=0,\ldots,n.\] 
\end{Def}
Note the isomorphisms $X_n\cong X_{n}^{*} \cong S^n(X_1)$ for all $n$, where $S^{n}$ denotes
the $n$th symmetric power. Let $\{x_0,x_1,\ldots, x_n\}$ be the set of
coordinate functions on a $n+1$ dimensional vector space $V_n$, so we consider
$X_{n}=V_{n}^{*}$. As $\kk$ is infinite, $\kk[V_n]$ can
be viewed as the polynomial ring $R_n:=S(X_{n})= \kk[x_0,x_1, \ldots ,x_n]$. If
$\Ga$ acts basically on $X_n$, one can then check that the induced action of
$\Ga$ on $R_n$ is given by the formula \[a*f = \exp(aD_n)f \qquad \text{ for
  all }a \in \Ga, \,\,\,f \in R_n,\] where $D_n$ is the \emph{Weitzenb\"{o}ck derivation} 
\[D_{n}:=x_{0}\frac{\partial}{\partial{x_{1}}}+\ldots+
x_{n-1}\frac{\partial}{\partial{x_{n}}} \quad \text{ on }R_{n}.\] Furthermore, the algebra of invariants $\kk[V_n]^{\Ga}$ is precisely the kernel of $D_n$. We denote this by $A_n$. The algebras $A_n$ have been objects of intensive study for well over a hundred years, owing to their connection with the classical invariants and covariants of binary forms. While they are known to be finitely generated by the Maurer-Weitzenb\"ock Theorem \cite{Weitzenbock}, the number of generators appears to increase rapidly with dimension, and explicit generating sets are (reliably!) known only for $n \leq 7$ (see also the table in section 2). In this article, we shall instead construct explicit separating sets for \emph{all} values of $n$.

This article is organised as follows: in Section 2 we state our main results, and explain briefly the connection between the algebras $A_n$ and the covariants of binary forms. In Section 3 we prove a crucial lemma on the radical of the Hilbert ideal of $A_n$ which may be of independent interest. Section 4 contains the main body of the proof of our result, while Section 5 is devoted to the proof of a technical lemma which is required in order to construct a separating set for $A_n$ when $n \equiv 0 \mod 4$.

Most of this work was completed during a visit of the first author to TU
M\"unchen in July 2010. We would like to thank Gregor Kemper for making this visit possible.

\section{Background and statement of results}\label{ResultsSection}

Let $U_{n}$ denote the $\kk$-vector space of binary forms of degree $n$, which
are homogeneous polynomials of the form $\sum_{i=0}^{n}a_iX^iY^{n-i}$ in the
variables $X$ and $Y$, $a_{i}\in \kk$. This is a vector space of dimension
$n+1$ with basis the set of monomials in $X$ and $Y$ of degree $n$. The
natural action of the group $G:=\SL_2(\kk)$ on a two dimensional vector space
with basis $\{X,Y\}$ induces an action of $G$ on the vector space
$U_n$. Classically speaking, an invariant is a polynomial in the coefficients
$a_i$ which is unchanged under the action of $G$ - in modern notation, an
element of $\kk[U_{n}]^{G}$. Note that the additive group $\Ga$ is embedded in
$G$ as the subgroup of matrices of the form 
$\left(\begin{array}{cc} 1&* \\ 0&1 \end{array}\right)$, 
and the subgroup $\Ga$ acts basically on $U_n$ (with respect to the basis
$\{\frac{1}{k!}X^{n-k}Y^{k}: k=0,\ldots,n\}$).
The pioneers of invariant theory also studied ``covariants'', which are
polynomials in both the coefficients $a_i$ and the variables $X$ and $Y$
themselves which are fixed under the action of $G$. In modern notation, the
algebra of covariants is $\kk[U_n\oplus U_1^{*}]^G$. There is, in fact, an
even stronger connection between covariants and the basic actions of $\Ga$:
the algebras $\kk[U_n\oplus U_1^{*}]^G$ and $\kk[U_n]^{\Ga}$ are actually
isomorphic. Let us identify the algebra $\kk[U_n \oplus U_1^{*}]$ with the
polynomial ring $\kk[a_0,a_1, \cdots, a_n,X,Y]$ (we abuse notation by using
the same letters $a_{i}$ for coordinates and coordinate functions). Define a mapping $\Phi:
\kk[U_n\oplus U_1^{*}]^G \rightarrow \kk[U_{n}]^{\Ga}$ by 
\begin{equation}\label{Roberts} \Phi(f(a_0,a_1,a_2, \ldots, a_n,X,Y)):= f(a_0,a_1,a_2, \ldots, a_n,0,1).\end{equation}
The theorem of Roberts \cite{Roberts} states that $\Phi$ is an isomorphism. In
classical invariant theory one often studies the basic actions of $\Ga$ in
order to get a handle on the covariants of binary forms using Roberts'
isomorphism. One word of caution is needed at the point. While  \cite[Proposition 1]{KohlsKraft}  implies 
that a separating set for $\kk[U_n \oplus U_1^{*}]^{\SL_2(\kk)}$ must be
mapped under $\Phi$ to a separating set for $\kk[U_{n}]^{\Ga}$, the converse
is not necessarily true, so the separating sets we construct in this paper
most likely do not lift to give  separating sets for the covariants of binary
forms (cf. \cite[Remark 3]{KohlsKraft}).

We now state our main results. For any real number $x$, the symbol $[x]$ denotes the largest integer less than or equal to $x$. We begin by definining some important invariants, namely
\begin{equation}\label{DefFm}
f_{m}:=\sum_{k=0}^{m-1}(-1)^{k}x_{k}x_{2m-k}+\frac{1}{2}(-1)^{m}x_{m}^{2}\in
\ker D_{n} \quad \text{ for } m=1,\ldots,[\frac{n}{2}]
\end{equation}
and $f_{0}:=x_{0}$. Further, we define the elements
\begin{equation}\label{DefSm}
s_{m}:=\sum_{k=0}^{m}(-1)^{k}\frac{2m+1-2k}{2}x_{k}x_{2m+1-k}\in R_{n} \quad
\text{ for } m=1,\ldots,[\frac{n-1}{2}]
\end{equation}
and $s_{0}:=x_{1}$, which satisfy  \[D_{n}s_{m}=f_{m}, \text{ for all } m,\]
and in particular $D_{n}s_{m} \in A_n\setminus\{0\}$. Elements with this
property are called \emph{local slices.}

For any $a\in R_{n}\setminus\{0\}$, let $\nu(a)$ denote the nilpotency index
$\nu(a):=\min\{m\in {\mathbb N}: D_{n}^{m+1}(a)=0\}$, and $\nu(0):=-\infty$. If
$s\in R_{n}$ is a local slice, then for any $a\in R_{n}$ we define
\begin{eqnarray*}
\epsilon_{s}(a)&:=&(\exp(tD_{n})a)|_{t:=-s/D_{n}s}\cdot (D_{n}s)^{\nu(a)}\\&=&\sum_{k=0}^{\nu(a)}\frac{(-1)^{k}}{k!}(D_{n}^{k}a)s^{k}(D_{n}s)^{\nu(a)-k}\in
A_{n}.
\end{eqnarray*}

By the Slice Theorem \cite[Corollary 1.22]{FreudenburgBook}, we have
\[
A_{n}\subseteq\kk[\epsilon_{s}(x_{0}),\ldots,\epsilon_{s}(x_{n})]_{D_{n}s}.
\]
When $s=x_1$ and so $D_{n}s=x_0$, this is the first stage in Lin Tan's (and
van den Essen's) algorithm for producing a generating set for $A_{n}$ \cite{TanAlgorithm,VanDenEssen}.

\begin{Theorem}\label{MainTheorem}
Given $n$, we define a set $E_n$ consisting of the following elements:

\[
f_{0},f_{1},\ldots,f_{[\frac{n}{2}]},
\]
\[
\epsilon_{s_{0}}(x_{2}),\ldots,\epsilon_{s_{0}}(x_{n}),
\]
\[
\epsilon_{s_{1}}(x_{1}),\ldots,\epsilon_{s_{1}}(x_{n}),
\]
\[
\epsilon_{s_{2}}(x_{2}),\ldots,\epsilon_{s_{2}}(x_{n}),
\]
\[
\epsilon_{s_{3}}(x_{3}),\ldots,\epsilon_{s_{3}}(x_{n}),
\]
\[
\vdots
\]
\[
\epsilon_{s_{[\frac{n-1}{2}]}}(x_{[\frac{n-1}{2}]}),\ldots,\epsilon_{s_{[\frac{n-1}{2}]}}(x_{n}).
\]

If $n \equiv 0 \mod 4$ we also append to $E_n$ an extra invariant $w$ which is
defined in Lemma \ref{TheBigLemma}. Then the set $E_n$ is a separating set for $A_n$.
\end{Theorem}     
                                     
Note that $\epsilon_{s_{0}}(x_{0})=f_{0}$ and $\epsilon_{s_{0}}(x_{1})=0$. The size of this separating set is about $\frac{3}{8}n^{2}$. The following table shows its exact size for some values of n. The lower line gives the size $c_{n}$
of a minimal generating set for $A_n$, see  Olver \cite[p. 40]{Olver}.  Olver says this list can not be trusted for
$n\ge 7$. For $c_{7}$, we use Bedratyuk's value $c_{7}=147$ \cite{Bedratyuk7}, while in Olver's
list values $c_{7}=124$ or $c_{7}=130$ are offered, depending on the source. 
We also want to remark that for $n\ge 5$, we could save $5$ elements by
replacing the $10$ elements of $E_{4}\setminus\{w\}$ appearing in $E_{n}$ by the $5$ generators of
$A_{4}$. (For $n=4$ we could save $6$ elements). Note that generators for $n\le 7$ are listed explicitly in \cite{BedratyukCasimir,Bedratyuk7}.
\\

\begin{tiny}
\begin{table}[h]\label{gentable}
{
\begin{tabular}{c|c|c|c|c|c|c|c|c|c|c|c|c|c|c|c|c|c}
n&4 &5& 6& 7& 8& 9& 10& 11& 12& 13& 14& 15& 16& 17& 18& 19& 20 \\
\hline
$|E_n|$& 11& 16& 20& 28& 34& 43& 49& 61& 69& 82& 90& 106& 116& 133& 143& 163& 175 \\
\hline
$c_{n}$& 5&23 & 26 & 147 &69 &415 &475 &949&?&?&?&?&?&?&?&?&?\\
\end{tabular}
}                                         
\end{table}
\end{tiny}
It is also worth noting that our separating set consists of invariants whose
degree is at most $2n+1$. 

\section{The radical of the Hilbert ideal}\label{HilbertIdealSection}

Let $R_n$, $D_n$ and $A_n$ be as in the introduction. For any $m<n$ we have the algebra homomorphism \[\pi_{m,n}: R_{n}\rightarrow
R_{m}, \,\quad 
f(x_{0},x_{1},\ldots,x_{n})\mapsto f(\underbrace{0,\ldots,0}_{n-m \text{
   times}},x_{0},\ldots,x_{m})\] which satisfies
$\pi_{m,n}\circ D_{n}=D_{m}\circ \pi_{m,n}$ and thus induces a map
$A_{n}\rightarrow A_{m}$. 

Consider the Hilbert ideal
$I_{n}:=A_{n,+}R_{n}\unlhd R_{n}$. 
With the invariants $f_{m}$ defined in \eqref{DefFm}, we get the
following inclusion for its radical:
\begin{equation}\label{HilbertlowerInclusion}
(x_{0},\ldots,x_{[\frac{n}{2}]})R_{n}=\sqrt{(f_{0},f_{1},\ldots,f_{[\frac{n}{2}]})R_{n}}\subseteq
\sqrt{I_{n}}.
\end{equation}
The main purpose of this section is to prove that the reverse inclusion holds too.

\begin{Prop}\label{IdealInclusion}
\begin{enumerate}
\item[(a)] The radical of the Hilbert ideal is given by
\[
\sqrt{I_{n}}=(x_{0},\ldots,x_{[\frac{n}{2}]})R_{n}.
\]
\item[(b)] $\pi_{n-[\frac{n}{2}]-1,n}(A_{n})=\kk$.
\item[(c)] $\pi_{m,2m}(A_{2m})=\kk[x_{0}^{2}]$ for $m$ odd.
\item[(d)] $\pi_{m,2m}(A_{2m})=\kk[x_{0}^{2},x_{0}^{3}]$ for $m$ even.
\end{enumerate}
\end{Prop}

\begin{proof}
We will make use of Roberts' isomorphism as defined in the previous section,
with the only difference that we will choose variables so that the $\Ga$-actions
become basic, using the notations of the introduction. 
Additionally, let $\Ga$ act basically on $\langle y_{0},y_{1}\rangle_{\kk}$. The action of
$\Ga$ on $\tilde{R}_n:= R_n[y_0,y_1]$ extends to an action of $\SL_2(\kk)$ on
$\tilde{R}_n$ such that the following holds:

\begin{enumerate}
\item for any $a\in \kk \setminus \{0\}$ and
$\mu_{a}:=\left(\begin{array}{cc} a^{-1}& 0\\ 0&a\end{array} \right)\in
\SL_{2}$ we have $\mu_{a}(x_{k})=a^{2k-n}x_{k}$ for $k=0,\ldots,n$ and
$\mu_{a}(y_{k})=a^{2k-1}y_{k}$ for $k=0,1$. 
\item for
$\tau:=\left(\begin{array}{cc} 0& -1\\ 1&0\end{array} \right)\in \SL_{2}$ we
have $\tau(x_{k})=(-1)^{k}\frac{(n-k)!}{k!}x_{n-k}$ for $k=0,\ldots,n$ and
$\tau(y_{k})=(-1)^{k}y_{1-k}$ for $k=0,1$.
\end{enumerate} 

Recall from section \ref{ResultsSection} that the algebra map
\[
\Phi: \tilde{R}_{n}\rightarrow R_{n}, \quad
f(x_{0},\ldots,x_{n},y_{0},y_{1})\mapsto f(x_{0},\ldots,x_{n},0,1)
\]
induces an isomorphism of invariant rings $\tilde{R}_{n}^{\SL_{2}}\rightarrow
R_{n}^{\Ga}=A_{n}$. Now let $f\in A_{n}$ and $F\in \tilde{R}_{n}^{\SL_{2}}$
with $\Phi(F)=f$. Then we also have $f=\Phi(\mu_{a}\cdot F)$ for all $a\in
\kk\setminus\{0\}$, i.e.
\[
f(x_{0},\ldots,x_{n})=F(a^{-n}x_{0},a^{-n+2}x_{1},\ldots,a^{n}x_{n},0,a) \,\text{ for all } \,\,a \in
\kk\setminus\{0\}.
\]
Thus,
\[
\pi_{n-[\frac{n}{2}]-1,n}(f)=F(0,\ldots,0,a^{2([\frac{n}{2}]+1)-n}x_{0},\ldots,a^{n}x_{n-[\frac{n}{2}]-1},0,a)  \text
\,\text{ for all } \,\,a \in
\kk\setminus\{0\},
\]
and since this equation is polynomial in $a$ and $\kk$ is an infinite field,
it also holds for $a=0$. Therefore,
$\pi_{n-[\frac{n}{2}]-1,n}(f)=F(0,\ldots,0)\in \kk$, which proves (a) and
(b). Similarly, for $n=2m$ we find
\[
\pi_{m,2m}(f)=F(0,\ldots,0,x_{0},a^{2}x_{1},\ldots,a^{2m}x_{m},0,a) \text{ for all } \,\,a \in
\kk\setminus\{0\},
\]
which is again polynomial in $a$. When $a=0$, we get
$$\pi_{m,2m}(f)=F(0,\ldots,0,x_{0},0,\ldots,0)=p(x_{0})$$ for some polynomial
$p(x_{0})\in\kk[x_{0}]$, so $\pi_{m,2m}(A_{2m})\subseteq \kk[x_{0}]$. Since
$\pi_{m,2m}(f_{m})=\frac{1}{2}(-1)^{m}x_{0}^{2}$, we get the inclusion
$\kk[x_{0}^{2}]\subseteq \pi_{m,2m}(A_{2m})$. Using that $F$ is also invariant
under $\tau$, we find in the same manner as before
\[
\pi_{m,2m}(f)=\pi_{m,2m}(\Phi(\tau\mu_{a^{-1}}F))|_{a=0}=F(0,\ldots,0,(-1)^{m}x_{0},0,\ldots,0)=p((-1)^{m}x_{0}).
\]
Therefore, for $m$ odd we get
$\pi_{m,2m}(f)=p(x_{0})=p(-x_{0})\in\kk[x_{0}^{2}]$, which proves (c). For
$m$ even, to prove (d), we refer to Lemma \ref{TheBigLemma}, which gives a $w\in A_{2m}$ with $\pi_{m,2m}(w)=x_{0}^{3}$, so
$\kk[x_{0}^{2},x_{0}^{3}]\subseteq \pi_{m,2m}(A_{2m})\subseteq
\kk[x_{0}]$. Since $x_{0}$ generates the degree one elements of $A_{2m}$ and $\pi_{m,2m}(x_{0})=0$, we
are done.  
\end{proof}

We want to mention here that the method of proof for Proposition
\ref{TheBigLemma} (a) also works for decomposable actions. Consider
\[R:=\kk[x_{0,1},\ldots,x_{n_{1},1},\ldots,x_{0,k},\ldots,x_{n_{k},k}]\] and
$D=D_{n_{1}}+\ldots+D_{n_{k}}$ with $D_{n_{i}}=x_{0,i}\frac{\partial}{\partial
x_{1,i}}+\ldots+x_{n_{i}-1,i}\frac{\partial}{\partial
x_{n_{i},i}}$. Using an algebra homomorphism $\pi$ which behaves on each
subalgebra $\kk[x_{0,i},\ldots,x_{n_{i},i}]$ as
$\pi_{n_{i}-[\frac{n_{i}}{2}]-1,n_{i}}$, we get with the same proof

\begin{Theorem}
The radical of the Hilbert ideal of $\ker D$ is given by \[(x_{0,1},\ldots,x_{[\frac{n_{1}}{2}],1},\ldots,x_{0,k},\ldots,x_{[\frac{n_{k}}{2}],k})R.\]
\end{Theorem}

\section{Construction of a separating set}

In this section, we prove our main result.

\begin{proof}[Proof of Theorem \ref{MainTheorem}]
For $V_{n}=\kk^{n+1}$ with $\kk[V_{n}]=R_{n}$, assume there are two elements
$a=(a_{0},\ldots,a_{n})$ and $b=(b_{0},\ldots,b_{n})$ of $V_{n}$ such that
$f(a)=f(b)$ for all $f\in E_{n}$. We have to show that $f(a)=f(b)$ for all $f\in
A_{n}$. As $x_{0}\in E_{n}$, we have $a_{0}=b_{0}$. Assume first $a_{0}=b_{0}\ne
0$. By the Slice Theorem, $A_{n}\subseteq \kk[E_{n}]_{x_{0}}$. Therefore, $f\in
A_{n}$ can be written as $f=\frac{p}{x_{0}^{l}}$ with $p\in \kk[E_{n}]$ and $l\ge
0$. By assumption, $p(a)=p(b)$ and $a_{0}^{l}=b_{0}^{l}\ne 0$, so
$f(a)=p(a)/a_{0}^{l}=p(b)/b_{0}^{l}=f(b)$. Now assume $a_{0}=b_{0}=0$ and
let $m$ be maximal such that $a_{0}=a_{1}=\ldots=a_{m}=0$, so $a_{m+1}\ne 0$
(if $m<n$). By induction on $k$,
we shall show that $b_{k}=0$ for $k=0,\ldots,\min\{m,[\frac{n}{2}]\}$. By assumption this holds
for $k=0$, so assume it holds for some $k<\min\{m,[\frac{n}{2}]\}$. Then
\begin{equation}\label{ak1}
\frac{(-1)^{k+1}}{2}    a_{k+1}^{2}=f_{k+1}(a)=f_{k+1}(b)=\frac{(-1)^{k+1}}{2}b_{k+1}^{2},
\end{equation}
so $b_{k+1}=0$ since $a_{k+1}=0$. If $m\ge [\frac{n}{2}]$, then $f(a)=f(0)=f(b)$ for any $f\in
A_{n}$ by Proposition \ref{IdealInclusion}, so now assume $0\le m
<[\frac{n}{2}]$. Equation \eqref{ak1} for $k=m$ shows
$0\ne a_{m+1}^{2}=b_{m+1}^{2}$. We now distinguish different cases.

{\it 1st Case: $m< [\frac{n-1}{2}]$}. Then $s_{m+1}$ is defined, and by the
Slice Theorem \[A_{n}\subseteq \kk[\epsilon_{s_{m+1}}(x_{0}),\ldots,\epsilon_{s_{m+1}}(x_{n})]_{f_{m+1}}.
\]
Applying $\pi:=\pi_{n-m-1,n}$ on both sides yields
\[
\pi(A_{n})\subseteq \kk[\pi(\epsilon_{s_{m+1}}(x_{m+1})),\ldots,\pi(\epsilon_{s_{m+1}}(x_{n}))]_{\pi(f_{m+1})},
\]
where we used that $\pi_{n-m-1,n}(\epsilon_{s_{m+1}}(x_{k}))=0$ for
$k=0,\ldots,m$. The right hand side is included in
$\kk[\pi(E_{n})]_{\pi(f_{m+1})}$. Therefore, for any $f\in A_{n}$ there is $p\in
\kk[E_{n}]$ and $l\ge 0$ such that $\pi(f)=\frac{\pi(p)}{\pi(f_{m+1})^{l}}$. Let
\[
\gamma: V_{n}\rightarrow V_{n-m-1}, \quad (c_{0},\ldots,c_{n})\mapsto(c_{m+1},\ldots,c_{n}).
\] 
Then
\begin{eqnarray*}
f(a)&=&\pi(f)(\gamma(a))=\frac{\pi(p)}{\pi(f_{m+1})^{l}}(\gamma(a))=\frac{\pi(p)(\gamma(a))}{(\pi(f_{m+1})(\gamma(a)))^{l}}=\frac{p(a)}{f_{m+1}(a)^{l}}\\
&=&
\frac{p(b)}{f_{m+1}(b)^{l}}=\frac{\pi(p)(\gamma(b))}{(\pi(f_{m+1})(\gamma(b)))^{l}}=\pi(f)(\gamma(b))=f(b).
\end{eqnarray*}
Here we used that the elements $p$ and $f_{m+1}$ of $E_{n}$ take the same value on
$a,b$ by assumption, and $f_{m+1}(a)=f_{m+1}(b)\ne 0$.

{\it 2nd Case: $[\frac{n-1}{2}]=m<[\frac{n}{2}]$}. In this case, $n$ has to be
even, and $n=2m'$ with $m'=m+1$. Let $\pi$ and $\gamma$ be as before, so
$\pi=\pi_{m',2m'}$ and $\gamma: V_{2m'}\rightarrow V_{m'}$. First assume $m'$
is odd. By Proposition \ref{IdealInclusion} (c) we have
\[
\pi(A_{n})=\kk[x_{0}^{2}]=\kk[\pi(f_{m'})].
\]
If $m'$ is even,  by Proposition \ref{IdealInclusion} (d) we have
\[
\pi(A_{n})=\kk[x_{0}^{2},x_{0}^{3}]=\kk[\pi(f_{m'}),\pi(w)],
\]
with $w$ the element of Lemma \ref{TheBigLemma}. In both cases, $\pi(A_{n})=\kk[\pi(E_{n})]$, so for any $f\in A_{n}$, there exists $p\in
\kk[E_{n}]$ such that $\pi(f)=\pi(p)$. Therefore,
\[
f(a)=\pi(f)(\gamma(a))=\pi(p)(\gamma(a))=p(a)=p(b)=\pi(p)(\gamma(b))=\pi(f)(\gamma(b))=f(b).
\]
We have shown: for any $f\in A_{n}$ we have $f(a)=f(b)$, and so we are done.
\end{proof}

\section{The existence of $w$.}\label{wSection}

In this section we prove Lemma~\ref{TheBigLemma}, which requires some more machinery. Note that we need this Lemma in order to construct a
separating set only in the case where $n \equiv 0 \mod 4$ --- in the other
cases, $w$ is not contained in our separating set. We will make use of
semitransvectants, which are the classical transvectants
transformed under Roberts' isomorphism, see for example \cite{Bedratyuk7,
  CayleyCollection, Olver}.
Recall that for a covariant
$F\in\tilde{R}_n^{\SL_{2}}=R_{n}[y_{0},y_{1}]^{\SL_{2}}$, its total degree in
$y_0,y_1$ is called the \emph{order} of $F$. For covariants $F$ and $G$ of orders $l$ and
$m$ respectively, we can construct new covariants given by
\[ \langle F, G \rangle^{(r)}:= \sum^r_{k=0} (-1)^k \left( \begin{array}{c} r
    \\ k \end{array} \right) \frac{\partial^{r}F}{\partial y_0^{r-k}\partial
  y_1^k}\frac{\partial^{r}G}{\partial y_0^{k}\partial y_1^{r-k}} \qquad r \leq
\min(l,m),\]which is called the $r$th \emph{transvectant} of $F$ and
$G$ (see \cite[p. 88]{Olver}). Transvectants play a key role in Gordan's famous proof of the finite
generation of covariants of binary forms \cite{GordanProvesfg}. The
  transformation of this construction under Roberts' isomorphism leads the
  following definition (see also \cite{BedratyukInvs}).

\begin{Def} 
Let $\Phi: R_{n}[y_{0},y_{1}]^{\SL_{2}}\rightarrow R_{n}^{\Ga}=A_{n}$ be Roberts'
isomorphism, given by substituting $y_{0}:=0,\,\, y_{1}:=1$.
Let $f$ and $g$ be a pair of invariants in $A_n$. Then for $r \leq \min(l,m)$, where $l$ and $m$ are the orders of $\Phi^{-1}(f)$ and $\Phi^{-1}(g)$ as above, we define the $r$th \emph{semitransvectant} of $f$ and $g$ by
\begin{equation*}  
[f,g]^{(r)}:= \Phi(\langle \Phi^{-1}(f),\Phi^{-1}(g) \rangle^{(r)}). 
\end{equation*}
\end{Def} 

\noindent In order to get an explicit expression for the semitransvectant, we introduce a second derivation on $R_n$, which is somewhat inverse to $D_n$:
\begin{equation*}
 \Delta_n:= \sum_{k=0}^n (n-k)(k+1)x_{k+1}\frac{\partial}{\partial x_k}.
\end{equation*}
This derivation comes from the other canonical embedding of $\Ga$ in $\SL_{2}$, namely for $f\in
R_{n}, \,\,a\in\kk$ we have $\left(\begin{array}{cc}1&
    \\a&1\end{array}\right)*f=\exp(a\Delta_{n})f$. Let $\ord(f)$ denote the
nilpotency index of $f$ with respect to $\Delta_{n}$. Assume $F\in
R_{n}[y_{0},y_{1}]^{\SL_{2}}$ is homogeneous  of
degree $m$ in the variables $y_{0},y_{1}$, so it can be written in the form
$F=fy_{1}^{m}+y_{0}\cdot(\dots)$ with $f\in R_{n}$. Then $\Phi(F)=f$, and
invariance of $F$ under the torus action implies all terms
$x_{0}^{a_{0}}\dots x_{n}^{a_{n}}$ in $f$ satisfy
$m=\sum_{k=0}^{n}(n-2k)a_{k}$. A polynomial $f\in R_{n}^{\Ga}$ with this property
is called \emph{isobaric} of \emph{weight} $m$, and then we have
$m=\ord(f)$. For an isobaric $f\in R_{n}^{\Ga}$, by
\cite[p.~43]{HilbertCourse} the inverse of Roberts' isomorphism is given by 
\[\Phi^{-1}(f) =
\sum^{\ord(f)}_{i=0}(-1)^{i}\frac{\Delta^i_n(f)}{i!}y_0^iy_1^{\ord(f)-i}.
\]

\begin{Prop}\label{transvectant}
Let $f,g \in R_n^{\Ga}$ be isobaric. Then for $r \leq \min(\ord(f),\ord(g))$,
the $r$th semitransvectant of $f$ and $g$ is given by the formula
\[ [f,g]^{(r)} = \sum_{k=0}^r (-1)^k\left( \begin{array}{c} r \\ k \end{array} \right) \Delta_n^{k}(f) \frac{(\ord(f)-k)!}{(\ord(f)-r)!}\Delta_n^{r-k}(g)\frac{(\ord(g)-r+k)!}{(\ord(g)-r)!} \]
\end{Prop}

\begin{proof}
 Let $\frac{\Delta^i_n(f)}{i!}:= \lambda_i$ and $\frac{\Delta^i_n(g)}{i!}:= \mu_i$. Then 
\[
\frac{ \partial^r \Phi^{-1}(f)}{\partial y_0^{r-k} \partial y_{1}^k} =
\sum_{i=r-k}^{\ord(f)-k} (-1)^{i}\lambda_i \frac{i!}{(i-r+k)!} \frac{(\ord(f)-i)!}{(\ord(f)-i-k)!}y_0^{i-r+k}y_1^{\ord(f)-i-k}
\] 
and 
\[
\frac{ \partial^r \Phi^{-1}(g)}{\partial y_0^{k} \partial y_1^{r-k}} =
\sum_{i=k}^{\ord(g)-r+k} (-1)^{i}\mu_i
\frac{i!}{(i-k)!} \frac{(\ord(g)-i)!}{(\ord(g)-i-r+k)!}
y_0^{i-k}y_1^{\ord(g)-i-r+k},
\]
therefore 
$$
\Phi\left(\frac{ \partial^r \Phi^{-1}(f)}{\partial y_0^{r-k} \partial y_{1}^{k}}\right) =(-1)^{r-k}
\lambda_{r-k}\frac{(r-k)!(\ord(f)-r+k)!}{(\ord(f)-r)!}
$$
 and 
$$
\Phi\left(\frac{ \partial^r \Phi^{-1}(g)}{\partial y_0^{k} \partial
    y_1^{r-k}}\right) =(-1)^{k}
\mu_{k}\frac{k!(\ord(g)-k)!}{(\ord(g)-r)!}.
$$ Using the fact that $\Phi$ is an algebra homomorphism we have
\begin{eqnarray*}
[f,g]^{(r)}& =& \sum_{k=0}^r (-1)^{k+r} \left(\begin{array}{c} r \\k
  \end{array}\right)  \lambda_{r-k}\frac{(r-k)!(\ord(f)-r+k)!}{(\ord(f)-r)!} \mu_{k}\frac{k!(\ord(g)-k)!}{(\ord(g)-r)!},
\\
& =& \sum_{k=0}^r (-1)^{k+r} \left(\begin{array}{c} r \\k \end{array} \right) \Delta_{n}^{r-k}(f)\frac{(\ord(f)-r+k)!}{(\ord(f)-r)!} \Delta_{n}^{k}(g)\frac{(\ord(g)-k)!}{(\ord(g)-r)!}\\
& =& \sum_{k=0}^r (-1)^{k} \left(\begin{array}{c} r \\k \end{array} \right) \Delta_{n}^{k}(f)\frac{(\ord(f)-k)!}{(\ord(f)-r)!} \Delta_{n}^{r-k}(g)\frac{(\ord(g)-r+k)!}{(\ord(g)-r)!}
\end{eqnarray*} as required.
\end{proof}

\begin{rem}
 This is analogous to \cite[Lemma~1]{BedratyukInvs}, using a different basis.
\end{rem}

The formula shows that,  up to some scalar factor,  $f_m$ (from \eqref{DefFm}) equals
$[x_0,x_0]^{(2m)}$ (while $[x_{0},x_{0}]^{(r)}=0$ for $r$ odd), and $\epsilon_{s_m}(x_1)$ equals $[x_0,f_m]^{(1)}$. We
wonder whether there is also a connection between $\epsilon_{s_m}(x_j)$ and
$[x_0,f^j_m]^{(j)}$. 
\begin{Lemma}\label{TheBigLemma} 
Suppose $n$ is divisible by 4, so $n=2m=4p$. Then there is an invariant $w \in A_n$ satisfying $\pi_{m,n}(w)=x_0^3$.
\end{Lemma}

\begin{proof}
Throughout the proof we use the shorthand $\pi:= \pi_{m,n}$, and we set
$f:=f_{p}=\frac{1}{2}\sum_{i=0}^m(-1)^ix_ix_{m-i}$ (which is proportional to $[x_0,x_0]^{(m)}$). Obviously, $f$ is
isobaric of weight $(n-2i)+(n-2(m-i))=2n-2m=n=\ord(f)$.
Thus we may define \[
\bar{w} := [x_0,f]^{(n)}=\sum_{k=0}^{n}(-1)^{k}\frac{n!^{2}k!}{(n-k)!}x_{k}\Delta_{n}^{n-k}(f)=\sum_{k=0}^{n}(-1)^{k}\frac{n!^{2}(n-k)!}{k!}x_{n-k}\Delta_{n}^{k}(f),
\]
where we used Proposition~\ref{transvectant}. Thus,
$$\pi(\bar{w}) = \sum_{k=0}^m (-1)^k \frac{n!^{2}(n-k)!}{k!} x_{m-k}\pi(\Delta_n^k(f)).$$   
Using Leibniz's formula for iterated differentiation of products, we have 
\[
\pi(\Delta_n^k(x_{i}x_{m-i})) = \sum_{j=0}^k \tiny{ \left( \begin{array}{c} k \\
      j \end{array} \right)}
\pi(\Delta_n^jx_i)\pi(\Delta_n^{k-j}x_{m-i})\]
\begin{eqnarray*}\label{deltaig}
\hspace{-0.6cm}
&=&\sum_{j=m-i}^{k-i}\tiny{\left(
    \begin{array}{c} k \\ j \end{array} \right)}
\frac{(i+j)!(n-i)!}{i!(n-i-j)!}\pi(x_{i+j})\frac{(m-i+k-j)!(m+i)!}{(m-i)!(m+i-k+j)!}\pi(x_{m-i+k-j})\\
& =& \sum_{j=0}^{k-m}\tiny{\left(\begin{array}{c}k \\ j+m-i \end{array}
  \right)}
\frac{(m+j)!(n-i)!}{i!(m-j)!}\pi(x_{m+j})\frac{(k-j)!(m+i)!}{(m-i)!(n-k+j)!}\pi(x_{k-j}).
\end{eqnarray*}
In particular, $\pi(\Delta^k_n(x_ix_{m-i}))=0$ for all $k<m$, and since $f$ is a linear combination of terms of the form $x_ix_{m-i}$, we have $\pi(\Delta^k_n(f))=0$ for all $k<m$.
From this, remembering $m$ is even, it follows that $\pi(\bar{w}) = n!^{2} x_0 \pi(\Delta^m_n(f))$. 
Therefore, since
$$\pi(\Delta_n^m(x_ix_{m-i})) = \left(\begin{array}{c} m \\ m-i \end{array} \right) x^2_0 \frac{(n-i)!}{i!}\frac{(m+i)!}{(m-i)!},$$
and $f=\frac{1}{2}\sum_{i=0}^m(-1)^ix_ix_{m-i}$, we obtain
\begin{eqnarray*}
\pi(\Delta_n^m(f))& =& \frac{1}{2}x_0^2\sum_{i=0}^m(-1)^i \left(
  \begin{array}{c} m \\ i \end{array} \right)
\frac{(n-i)!}{i!}\frac{(m+i)!}{(m-i)!}\\
&=& 
 \frac{((2p)!)^2}{2}x_0^2\sum_{i=0}^{2p}(-1)^i \left( \begin{array}{c} 2p \\ i
   \end{array} \right)\left( \begin{array}{c} 4p-i \\ 2p \end{array}
 \right)\left( \begin{array}{c} 2p+i \\ i \end{array} \right).
\end{eqnarray*} 
Thus, $\pi(\bar{w})=n!^{2} x_0 \pi(\Delta^m_n(f))$ is a nonzero multiple of
$x_{0}^{3}$ if the sum above is nonzero. This follows from  Lemma \ref{TheSum}.\end{proof}

\begin{rem} 
With $g:=\Delta_{n}^{n}(f)$, we have
$\bar{w}=c\cdot\sum_{k=0}^{n}(-1)^{k}x_{k}D^{k}g$ with $c\in \kk$.
\end{rem}

\begin{Lemma}\label{TheSum}
For all $p\ge 1$ we have 
\begin{equation*} 
\sum^{2p}_{k=0} (-1)^k\left( \begin{array}{c} 2p \\ k \end{array} \right)\left( \begin{array}{c} 4p-k \\ 2p \end{array} \right)\left( \begin{array}{c} 2p+k \\ k \end{array} \right)=(-1)^p\frac{(3p)!}{(p!)^3}.
\end{equation*}
\end{Lemma}

\begin{proof}
The argument which follows was produced using the implementation of Zeilberger's algorithm \cite{WilfZeilberger} in the remarkable EKHAD package for Maple \cite{AEqualsB}.
Let $$F(p,k):= \left\{ \begin{array}{cc}(-1)^k\left( \begin{array}{c} 2p \\ k
      \end{array} \right)\left( \begin{array}{c} 4p-k \\ 2p \end{array}
    \right)\left( \begin{array}{c} 2p+k \\ k \end{array} \right) & 0 \leq k
    \leq 2p\\ 0 & \text{Otherwise},\end{array} \right.$$ and let $S(p):=
\sum_{k=0}^{2p}F(p,k)$. We show  that the following recurrence relation holds:
\begin{equation}\label{recurrence} 6(3p+2)(3p+1)S(p)+2(p+1)^2S(p+1)=0.\end{equation}

To do this, we consider the function
\begin{eqnarray*}G(p,k):=  \frac{1}{2}k^2(180-184k+1036p+59k^2+2192p^2-790pk-1116p^2k+168pk^2 \\  +2024p^3-8k^3+688p^4+k^4-520p^3k+120p^2k^2-10pk^3)(-4p+k-1)\frac{F(p,k)}{R(p,k)} \end{eqnarray*}
where $R(p,k) = (2p+1)(-2p-2+k)^2(-2p-1+k)^2$. For $0\le k\le 2p-1$ it satisfies the relation
\[G(p,k+1)-G(p,k)=6(3p+2)(3p+1)F(p,k)+2(p+1)^2F(p+1,k).\]
 Summing both sides over $0\le k\le 2p-1$ (and adding remaining terms) produces (\ref{recurrence}), and an easy inductive argument then shows that $S(p)= (-1)^p\frac{(3p)!}{(p!)^3}$.
\end{proof}
\begin{rem} The sum $S(p)$ is the well-poised hypergeometric series 
 \[ \left( \begin{array}{c} 4p \\ 2p \end{array} \right) \sum^{\infty}_{k=0} \frac{(-2p)_k(2p+1)_k(-2p)_k}{(1)_k(-4p)_k k!} = \left( \begin{array}{c} 4p \\ 2p \end{array} \right) {}_{3}F_{2}[-2p,2p+1,-2p;1,-4p; 1].\] Surprisingly, the series is not summable by any classical hypergeometric sum theorem (e.g. Dixon's theorem, Watson's theorem) because the series  $_{3}F_{2}[-2p,2p+1,-2p;1,-4p; z]$, $p$ not an integer, does not converge when $z=1$, see \cite[Chapter~2]{SlaterHypergeom}. For this reason, we have to apply Zeilberger's algorithm for partial sums in order to sum the series. In the language of WZ-theory, $\bar{F},G$ is a WZ-pair, where $\bar{F}(p,k):= (-1)^p\frac{(p!)^3F(p,k)}{(3p)!}$, and $R(p,k)$ is the corresponding WZ-proof certificate.   
\end{rem}

\bibliographystyle{plain}
\bibliography{MyBib}
\end{document}